\documentclass[a4paper,11pt]{article}
\usepackage{titlesec}
\usepackage{amsthm}
\usepackage{amssymb}
\usepackage{color}
\usepackage{mathtools}
\usepackage{geometry}

\usepackage{algorithm2e}
\usepackage[T1]{fontenc}

\usepackage{verbatim}
\numberwithin{equation}{section}

\newtheorem{theorem}{Theorem}[section]

\theoremstyle{definition}

\newtheorem{remark}[theorem]{Remark}
\usepackage{array}
\newcolumntype{C}[1]{>{\centering\arraybackslash}m{#1}}
\usepackage{graphics}
\usepackage{amssymb, latexsym, amsmath, amsfonts, epsfig}
\usepackage{bm}
\usepackage{graphicx}
\usepackage{color}
\usepackage{epstopdf}
\usepackage{comment}
\usepackage{soul}
\usepackage[normalem]{ulem}
\usepackage{float}
\usepackage{MnSymbol}

\usepackage{caption}
\date{\vspace{-6ex}}

\begin{document}

\newcommand{\Question}[1]{{\marginpar{\color{blue}\footnotesize #1}}}
\newcommand{\blue}[1]{{\color{blue}#1}}
\newcommand{\red}[1]{{\color{red} #1}}

\newif \ifNUM \NUMtrue

\title{Analytical solutions to some generalized and polynomial eigenvalue problems}
\author{Quanling Deng\thanks{Department of Mathematics, University of Wisconsin--Madison, Madison, WI 53706, USA. E-mail addresses: quanling.deng@math.wisc.edu; qdeng12@gmail.com. Online open access: \textit{Special Matrices}, 9(1), 240-256. https://doi.org/10.1515/spma-2020-0135.}
}

\maketitle

\begin{abstract}
It is well-known that the finite difference discretization of the Laplacian eigenvalue problem $-\Delta u = \lambda u$ leads to a matrix eigenvalue problem (EVP) $A x= \lambda x$ where the matrix $A$ is Toeplitz-plus-Hankel. Analytical solutions to tridiagonal matrices with various boundary conditions are given in a recent work of Strang and MacNamara. 
We generalize the results and develop analytical solutions to certain generalized matrix eigenvalue problems (GEVPs) $A x= \lambda Bx$ which arise from the finite element method (FEM) and isogeometric analysis (IGA). The FEM matrices are corner-overlapped block-diagonal while the IGA matrices are almost Toeplitz-plus-Hankel. 
In fact, IGA with a correction that results in Toeplitz-plus-Hankel matrices gives a better numerical method. 
In this paper, we focus on finding the analytical eigenpairs to the GEVPs while developing better numerical methods is our motivation. 
Analytical solutions are also obtained for some polynomial eigenvalue problems (PEVPs). Lastly, we generalize the eigenvector-eigenvalue identity (rediscovered and coined recently for EVPs) for GEVPs and derive some trigonometric identities. 
\end{abstract}
%


\section{Introduction} \label{sec:intr}

It is well-known that the following tridiagonal Toeplitz  matrix
\begin{equation*}
A = 
\begin{bmatrix}
2 & -1 &  \\
-1 & 2 & -1 &  \\
   & \ddots & \ddots & \ddots &  \\
  &  & -1 & 2 & -1  \\
&& & -1 & 2  \\
\end{bmatrix}_{n\times n}
\end{equation*}
has analytical eigenpairs $(\lambda_j, x_j)$ with $x_j = (x_{j,1}, \cdots, x_{j, n})^T$ where
\begin{equation*}
\lambda_j = 2  - 2 \cos(j \pi h), \quad x_{j,k} = c \sin(j\pi k h), \quad h = \frac{1}{n+1}, \quad  j, k =1,2,\cdots, n
\end{equation*}
with $0 \ne c \in \mathbb{R}$ (we assume that $c$ is a nonzero constant throughout the paper); see, for example, \cite[p. 514]{meyer2000matrix} for a general case of tridiagonal Toeplitz matrix. Following the constructive technique in \cite[p. 515]{meyer2000matrix} that finds the analytical solutions to the matrix eigenvalue problem (EVP) $A x= \lambda x$, one can derive analytical solutions to the generalized matrix eigenvalue problem (GEVP) $A x= \lambda B x$ where $B$ is an invertible tridiagonal Toeplitz matrix. 
For example, let
\begin{equation*}
B = 
\begin{bmatrix}
 \frac{2}{3} & \frac{1}{6} \\[0.2cm]
\frac{1}{6} & \frac{2}{3} & \frac{1}{6} \\[0.2cm]
 & \ddots & \ddots & \ddots &  \\[0.2cm]
&  & \frac{1}{6} & \frac{2}{3} & \frac{1}{6}  \\[0.2cm]
&& & \frac{1}{6} & \frac{2}{3}  \\
\end{bmatrix}_{n\times n},
\end{equation*}
then, the GEVP $A x= \lambda B x$ has analytical eigenpairs $(\lambda_j, x_j)$ with $x_j = (x_{j,1}, $ $ \cdots,  x_{j, n})^T$ where (see \cite[Sec. 4]{hughes2008duality} for a scaled case; $A$ is scaled by $1/h$ while $B$ is scaled by $h$)
\begin{equation*}
\lambda_j = - 6  + \frac{18}{2 + \cos(j \pi h)}, \quad x_{j,k} = c \sin(j\pi k h), \quad h = \frac{1}{n+1}, \quad  j, k =1,2,\cdots, n.
\end{equation*}

These matrices or their scaled (by constants) versions arise from various applications. For example, the matrix $A$ arises from the discrete discretizations of the 1D Laplace operator by the finite difference method (FDM, cf., \cite{smith1985numerical,strang1988linear}, scaled by $1/h^2$) or the spectral element method (SEM, cf., \cite{patera1984spectral}, scaled by $1/h$). The work \cite{strang2014functions} showed that functions of the (tridiagonal)  matrices that arise from finite difference discretization of the Laplacian with different boundary conditions are Toeplitz-plus-Hankel. Similar analytical results for tridiagonal matrices from finite difference discretization are obtained in \cite{chang2009exact}. 
 The matrix $B$ (and also $A$) arises from the discrete discretization by the finite element method (FEM, cf., \cite{ciarlet2002finite,hughes2012finite,strang1973analysis}, scaled by $h$).

In general, it is difficult to find analytical solutions to the EVPs. The work by Losonczi \cite{losonczi1992eigenvalues} gave analytical eigenpairs to the EVPs for some symmetric and tridiagonal matrices. A new proof of these solutions was given by da Fonseca \cite{da2007eigenvalues}. 
%
%
We refer to the recent work \cite{da2019eigenpairs} for a survey on the analytical eigenpairs of tridiagonal matrices. 
For pentadiagonal and heptadiagonal matricies, finding analytical eigenpairs becomes more difficult. The work \cite{barrera2017asymptotics} derived asymptotical results of eigenvalues for pentadiagonal symmetric Toeplitz matrices. Some spectral properties were found in \cite{fasino1988spectral} for some pentadiagonal symmetric matrices. A recent work by An{\dj}eli\'c and da Fonseca \cite{andjelic2020some} presents some determinantal considerations for pentadiagonal matrices. The work \cite{solary2013finding} gave analytical eigenvalues (as the zeros of some complicated functions) for heptadiagonal symmetric Toeplitz matrices.

To the best of the author's knowledge, there are no widely-known articles in literature which address the issue of finding analytical solutions to either the EVPs with more general matrices (other than the ones mentioned above) or the GEVPs. The articles \cite{hughes2008duality,hughes2014finite,calo2019dispersion} present analytical solutions (somehow implicitly) to GEVPs for some tridiagonal and/or pentadiagonal matrices that arise from the isogeometric analysis (IGA) of the 1D Laplace operator. For heptadiagonal and more general matrices, no analytical solutions exist and the numerical approximations in \cite{hughes2014finite,calo2019dispersion,deng2018dmm} can be considered as asymptotic results for certain structured matrices arising from the numerical discretizations of the differential operators.

In this paper, we present analytical solutions to GEVPs for mainly two classes of matrices. The first class is the Toeplitz-plus-Hankel matrices while the second class is the corner-overlapped block-diagonal matrices. The main insights for the Toeplitz-plus-Hankel matrices are from the numerical spectral approximation techniques where a particular solution form such as, the well-known Bloch wave form, is sought. For the corner-overlapped block-diagonal matrices, we propose to decompose the original problem into a lower two-block matrix problem where one of the blocks is a quadratic eigenvalue problem (QEVP). We solve the QEVP by rewriting the problem and applying the analytical results from the tridiagonal Toeplitz matrices. 
Generalization of finding solutions to polynomial eigenvalue problem (PEVP) is also given.
Additionally, Denton, Parke, Tao, and Zhang in a recent work \cite{denton2021eigenvectors} rediscovered and coined the eigenvector-eigenvalue identity for certain EVPs. We generalize this identity for the GEVPs. Based on these identities, we derive some interesting trigonometric identities.

The rest of the article is organized as follows.
Section~\ref{sec:mr} presents the main results, i.e., the analytical solutions to the two classes of matrices. Several examples are given and discussed. Section \ref{sec:ti} generalizes the eigenvector-eigenvalue identity and derives some trigonometric identities. 
Potential applications to the design of better numerical discretization methods for partial differential equations and other concluding remarks are presented in Section~\ref{sec:conclusion}.

\section{Main results} \label{sec:mr}

\subsection{Notation and problem statement}
Throughout the paper, we denote matrices and vectors by uppercase and lowercase letters, respectively. 
In particular, let $A \in \mathbb{C}^{n\times n}$ be a square matrix with entries denoted as $A_{jk},j,k=1,\cdots,n,$ and $x\in \mathbb{C}^{n}$ be a vector with entries denoted as $x_j,j=1,\cdots,n,$ in the complex field. 
We denote by $A^{(\alpha,m)}$ a matrix that the entries depend on a sequence of parameters $\alpha = (\alpha_0, \alpha_1, \cdots, \alpha_m).$
The superscript $\cdot^{(\alpha,m)}$ is omitted when the context is clear. 
We denote by $T^{(\alpha,m)} = (T^{(\alpha,m)}_{j,k}) \in \mathbb{C}^{n\times n}$ a symmetric and diagonal-structured Toeplitz matrix with entries
\begin{equation} \label{eq:T}
T^{(\alpha,m)}_{j,j+k} = 
\begin{cases}
\xi_{|k|}, \quad & |k| \le m, \quad k \in \mathbb{Z}, \quad  j = 1,\cdots, n, \\
0, \quad & \text{otherwise},
\end{cases}
\end{equation}
where $m\le n-1$ specifies the matrix bandwidth. 
Explicitly, this matrix can be written as
\begin{equation*}
\begin{aligned}
T^{(\alpha,m)} & = 
\begin{bmatrix}
\alpha_0   & \alpha_1   & \cdots   & \alpha_m  \\
\alpha_1    & \alpha_0 & \alpha_1  & \cdots  & \alpha_m   \\
\vdots  & \alpha_1  & \alpha_0 & \alpha_1  & \cdots   & \alpha_m   \\
\alpha_m  & \cdots  & \alpha_1  & \alpha_0 & \alpha_1  & \cdots   & \alpha_m   \\
 &\ddots & \ddots & \ddots & \ddots & \ddots & \ddots & \ddots  \\
&  & \alpha_m  & \cdots  &\alpha_1  & \alpha_0 & \alpha_1 & \cdots \\
&& &  \alpha_m  & \cdots  & \alpha_1  & \alpha_0 & \alpha_1   \\
& && &  \alpha_m  & \cdots  & \alpha_1  & \alpha_0    \\
\end{bmatrix}_{n\times n},
\end{aligned}
\end{equation*}
where the empty spots are zeros.
For a Hankel matrix, it appears to be different in each of the cases we consider in the paper
and we define it at its occurrence in the context.
\color{black}

For matrices $A, B \in \mathbb{C}^{n\times n}$, the EVP is to find the eigenpairs $(\lambda \in \mathbb{C}, x \in \mathbb{C}^n)$ such that
\begin{equation} \label{eq:evp}
A x= \lambda x
\end{equation}
while 
the GEVP is to find the eigenpairs $(\lambda \in \mathbb{C}, x \in \mathbb{C}^n) $ such that
\begin{equation} \label{eq:GEVP}
A x= \lambda B x.
\end{equation}

Throughout the paper,
we focus on GEVPs and since
the analytical eigenpairs for GEVPs with small dimensions are relatively easy to find,  we assume that 
the dimension $n$ is large enough for the generalization of matrices to make sense. 
 For simplicity, we slightly abuse the notation such as $A$ for a matrix and $(\lambda, x)$ for an eigenpair.  Once analytic eigenpairs for a GEVP  are found, eigenpairs for EVP $A x= \lambda x$ follows naturally by setting $B$ as an identity matrix.

\subsection{Toeplitz-plus-Hankel matrices}

In this section, we present analytical solutions to certain Toeplitz-plus-Hankel  matrices. 
The main insights are from the proof in finding analytical solutions to a tridiagonal Toeplitz matrix in \cite[p. 514]{meyer2000matrix} 
and the numerical spectral approximations of the Laplace operator (cf.,\cite{strang2014functions,hughes2008duality,hughes2014finite,calo2019dispersion}).
The main idea is to seek eigenvectors in a particular form such as the Bloch wave form $e^{a + \iota b}$, where $\iota^2=-1$ as in \cite{meyer2000matrix,hughes2008duality} and sinusoidal form as in \cite{hughes2014finite}. Our main contribution is the generalization of these results to GEVPs, especially, with larger matrix bandwidths.

We denote by $H^{(\alpha,m)} = (H^{(\alpha,m)}_{j,k}) \in \mathbb{C}^{n\times n}$ a Hankel matrix with entries
\begin{equation} \label{eq:h1}
H^{(\alpha,m)}_{j,k} = H^{(\alpha, m)}_{n-j+1,n-k+1} = 
\begin{cases}
 \alpha_{j+k}, \quad & k = 1, \cdots, m-j, \quad j =1, \cdots, m-1, \quad 2\le m\le n, \\
0, \quad & \text{otherwise},
\end{cases}
\end{equation}
which can be explicitly written as
\begin{equation*}
\begin{aligned}
H^{(\alpha,m)} & = 
\begin{bmatrix}
\alpha_2   & \alpha_3 & \alpha_4   & \cdots   & \alpha_m & 0 & \cdots \\
\alpha_3    & \alpha_4 &  \cdots  & \alpha_m & 0 & \ddots  \\
\alpha_4  &  \cdots  & \alpha_m & 0 & \ddots  \\
\vdots  & \alpha_m  &0 & \ddots \\
\alpha_m &0  &\ddots  \\
0 &\ddots  \\
\vdots  \\
\end{bmatrix}_{n\times n},
\end{aligned}
\end{equation*}
where the entries at the bottom-right corner are defined such that the matrix is persymmetric. 
Herein, a persymmetric matrix is defined as a square matrix which is symmetric with respect to the anti-diagonal. 
\color{black}
With this in mind, we give the following result.

\begin{theorem}[Analytical eigenvalues and eigenvectors, set 1]\label{thm:set1}
Let $n\ge2, 1\le m \le n-1$ and
$
A = T^{(\alpha,m)} - H^{(\alpha,m)}, B = T^{(\beta,m)} - H^{(\beta,m)}
$ 
with $T^{(\xi,m)}$ and $H^{(\xi,m)}$, $\xi = \alpha, \beta$,  defined in \eqref{eq:T} and \eqref{eq:h1}, respectively.
Assume that $B$ is invertible. Then, the GEVP \eqref{eq:GEVP} has eigenpairs $(\lambda_j, x_j)$ where
\begin{equation} \label{eq:set1}
\lambda_j = \frac{ \alpha_0 + 2 \sum_{l=1}^{m} \alpha_l \cos(l j\pi h) }{ \beta_0 + 2 \sum_{l=1}^{m} \beta_l \cos(l j\pi h)  }, \quad x_{j,k} = c \sin( j \pi k h), \quad h = \frac{1}{n+1}, \  j, k =1,2,\cdots, n.
\end{equation}
\end{theorem}

\begin{proof}
Following \cite[p. 515]{meyer2000matrix}, one seeks eigenvectors of the form $c\sin( j \pi k h)$. Using the trigonometric identity $\sin(\phi \pm \psi) = \sin(\phi) \cos(\psi) \pm \cos(\phi) \sin(\psi)$, one can verify that each row of the GEVP \eqref{eq:GEVP}, $\sum_{k=1}^n A_{ik} x_{j,k} = \lambda \sum_{k=1}^n B_{ik} x_{j,k}, i=1,\cdots, n$,  reduces to $\alpha_0 + 2 \sum_{l=1}^{m} \alpha_l \cos(l j\pi h) = \lambda \big( \beta_0 + 2 \sum_{l=1}^{m} \beta_l \cos(l j\pi h) \big)$, which is independent of the row number $i$. Thus, the eigenpairs $(\lambda_j, x_j)$ given in \eqref{eq:set1} satisfies \eqref{eq:GEVP}. 
The GEVP has at most $n$ eigenpairs and the $n$ eigenvectors are linearly independent. This completes the proof.
\end{proof}

We remark that all the matrices defined above are symmetric and persymmetric. 
For $m=3$ and $n\ge4$, the matrix $A=T^{(\xi,m)} - H ^{(\xi,m)}$ is of the form
\begin{equation*}
\begin{aligned}
A & = 
\begin{bmatrix}
\alpha_0 - \alpha_2  & \alpha_1 - \alpha_3  & \alpha_2   & \alpha_3  \\
\alpha_1 - \alpha_3   & \alpha_0 & \alpha_1  & \alpha_2  & \alpha_3   \\
\alpha_2  & \alpha_1  & \alpha_0 & \alpha_1  & \alpha_2   & \alpha_3   \\
\alpha_3  & \alpha_2  & \alpha_1  & \alpha_0 & \alpha_1  & \alpha_2   & \alpha_3   \\
&\ddots & \ddots & \ddots & \ddots & \ddots & \ddots & \ddots  \\
&  & \alpha_3  & \alpha_2  &\alpha_1  & \alpha_0 & \alpha_1 & \alpha_2 \\
&& &  \alpha_3  & \alpha_2  & \alpha_1  & \alpha_0 & \alpha_1 - \alpha_3  \\
& && &  \alpha_3  & \alpha_2  & \alpha_1 - \alpha_3 & \alpha_0 - \alpha_2   \\
\end{bmatrix}_{n\times n},
\end{aligned}
\end{equation*}
where the  Hankel matrix  $H^{(\xi,m)}$ modifies the matrix at the first and last few rows. 
For $m=1$, both $A$ and $B$ are tridiagonal Toeplitz matrices. For $m\ge 2$, both $A$ and $B$ are Toeplitz-plus-Hankel matrices. This result generalizes the finite-difference matrix results in Strang and MacNamara in \cite{strang2014functions}.

We present the following example. Let $m=2, \alpha_0 = 1, \alpha_1 = -1/3, \alpha_2 = -1/6,$ $\beta_0 = 11/20, \beta_1 = 13/60, \beta_2 = 1/120.$ Then, we have the following matrices
\begin{equation*} 
\begin{aligned}
A & = 
\begin{bmatrix}
\frac{7}{6} & -\frac{1}{3} & -\frac{1}{6} \\[0.2cm]
 -\frac{1}{3} & 1 & -\frac{1}{3} & -\frac{1}{6} \\[0.2cm]
  -\frac{1}{6} & -\frac{1}{3} & 1 & -\frac{1}{3} & -\frac{1}{6} \\[0.2cm]
   & \ddots & \ddots & \ddots & \ddots & \ddots &  \\[0.2cm]
  & & -\frac{1}{6} & -\frac{1}{3} & 1 & -\frac{1}{3} & -\frac{1}{6} \\[0.2cm]
  & &  & -\frac{1}{6} & -\frac{1}{3} & 1 & -\frac{1}{3}  \\[0.2cm]
&&& & -\frac{1}{6} & -\frac{1}{3} & \frac{7}{6}  \\
\end{bmatrix},
B &=
\begin{bmatrix}
\frac{13}{24} & \frac{13}{60} & \frac{1}{120} \\[0.2cm]
\frac{13}{60} & \frac{11}{20} & \frac{13}{60} & \frac{1}{120} \\[0.2cm]
  \frac{1}{120} & \frac{13}{60} & \frac{11}{20} & \frac{13}{60} & \frac{1}{120} \\[0.2cm]
   & \ddots & \ddots & \ddots & \ddots & \ddots &  \\[0.2cm]
  & & \frac{1}{120} & \frac{13}{60} & \frac{11}{20} & \frac{13}{60} & \frac{1}{120} \\[0.2cm]
  & &  & \frac{1}{120} & \frac{13}{60} & \frac{11}{20} & \frac{5}{24} \\[0.2cm]
&&& & \frac{1}{120} & \frac{13}{60} & \frac{13}{24}  \\
\end{bmatrix}.
\end{aligned}
\end{equation*}

The eigenpairs of the GEVP \eqref{eq:GEVP} with these matrices are
\begin{equation} \label{eq:evpm2}
\lambda_j = -20 + \frac{240(3 + 2\cos(j \pi h) )}{33 + 26\cos(j\pi h) + \cos(2j \pi h)}, \quad x_{j,k} = c \sin(j\pi k h), \quad j, k =1,2,\cdots, n.
\end{equation}

If we scale the matrix $A$ by $1/h$ while the matrix $B$ by $h$, then the new system is a GEVP that arises from the IGA approximation of the  Laplacian eigenvalue problem $-\Delta u = \lambda u$ on $[0,1]$; see, for example, \cite[eqns. 118--123]{hughes2014finite} (the first two rows have slightly different entries $A_{11} = 4/3, A_{12} = A_{21} = -1/6, B_{11} = 1/3, B_{12} = B_{21} = 5/24$). 

Similarly, if we define a Hankel matrix  $H^{(\alpha,m)} = (H^{(\alpha,m)}_{j,k}) \in \mathbb{C}^{n\times n}$ with entries
\begin{equation} \label{eq:h2}
H^{(\alpha,m)}_{j,k} = H^{(\alpha, m)}_{n-j+1,n-k+1} = 
\begin{cases}
 \alpha_{j+k-1}, \quad &  k = 1, \cdots, m-j+1, \quad j =1, \cdots, m, \quad 1\le m\le n, \\
0, \quad & \text{otherwise},
\end{cases}
\end{equation}
one can shift the phase by a half and seek solutions of the form $\sin\big( j \pi (k-\frac{1}{2}) h \big)$. As a consequence, we have the following result. 
\color{black}

\begin{theorem}[Analytical eigenvalues and eigenvectors, set 2] \label{thm:set2}
 Let $n\ge2, 1\le m \le n-1$ and
$
A = T^{(\alpha,m)} - H^{(\alpha,m)}, B = T^{(\beta,m)} - H^{(\beta,m)}
$ 
with $T^{(\xi,m)}$ and $H^{(\xi,m)}$, $\xi = \alpha, \beta$,  defined in \eqref{eq:T} and \eqref{eq:h2}, respectively.

Assume that $B$ is invertible.  Then, the GEVP  \eqref{eq:GEVP}  has eigenpairs $(\lambda_j, x_j)$  where
\begin{equation}  \label{eq:set2}
\lambda_j = \frac{ \alpha_0 + 2 \sum_{l=1}^{m} \alpha_l \cos(l j\pi h) }{ \beta_0 + 2 \sum_{l=1}^{m} \beta_l \cos(l j\pi h)  }, \quad x_{j,k} = c \sin\big( j \pi (k-\frac{1}{2}) h \big), \ h = \frac{1}{n}, \ j, k =1,2,\cdots, n.
\end{equation}
\end{theorem}

Herein, as an example, with $m=2$ and $n\ge 3$, the matrix $T^{(\alpha,m)} - H^{(\alpha,m)}$ is of the form
\begin{equation*}
\begin{bmatrix}
\alpha_0 - \alpha_1  & \alpha_1 - \alpha_2  & \alpha_2   \\
\alpha_1 - \alpha_2  & \alpha_0 & \alpha_1  & \alpha_2   \\
\alpha_2  & \alpha_1  & \alpha_0 & \alpha_1  & \alpha_2   \\
& \ddots & \ddots & \ddots & \ddots & \ddots  \\
&  & \alpha_2  &\alpha_1  & \alpha_0 & \alpha_1 - \alpha_2   \\
&& & \alpha_2  & \alpha_1 - \alpha_2  & \alpha_0 - \alpha_1   \\
\end{bmatrix}_{n\times n}.
\end{equation*}
The proof can be established similarly and we omit it for brevity. 
An example is the GEVP with the matrices in \cite[eqns. 118--123]{hughes2014finite}. These matrices satisfy the assumption of the above theorem and the analytical solutions are given by \eqref{eq:set2} with a scale $n^2$ to the eigenvalues. The eigenvectors of Theorems \ref{thm:set1} and \ref{thm:set2} correspond to the solutions of the  Laplacian eigenvalue problem on the unit interval. When $m\ge 3$, the results in Theorems \ref{thm:set1} and \ref{thm:set2} provide an insight to remove the outliers in high-order IGA approximations. 
We refer to \cite{hughes2008duality,hughes2014finite} for the outlier behavior in the approximate  spectrum 
and to \cite{deng2020boundary,deng2021outlier} for their eliminations.
Similarly, 
 we define a Hankel matrix  $H^{(\alpha,m)} = (H^{(\alpha,m)}_{j,k}) \in \mathbb{C}^{n\times n}$ with entries
\begin{equation} \label{eq:h3}
H^{(\xi,m)}_{j+1,k+1}  = H^{(\xi,m)}_{n-j,n-k} = 
\begin{cases}
-\xi_0/2, \quad &  j,k = 0, \\
\xi_{j+k}, \quad &  k = 1, \cdots, m-j, \ j =1, \cdots, m-1, m\ge2, \\
0, \quad & \text{otherwise}.
\end{cases}
\end{equation}
With the insights from the numerical methods for the Neumann eigenvalue problem on the unit interval, we have the following two sets of analytical eigenpairs.

\begin{theorem}[Analytical eigenvalues and eigenvectors, set 3] \label{thm:set3}

Let $n\ge2, 1\le m \le n-1$ and 
$
A = T^{(\alpha,m)} + H^{(\alpha,m)}, B = T^{(\beta,m)} + H^{(\beta,m)}
$ 
with $T^{(\xi,m)}$ and $H^{(\xi,m)}$, $\xi = \alpha, \beta$,  defined in \eqref{eq:T} and \eqref{eq:h3}, respectively.
Assume that $B$ is invertible.  Then, the GEVP  \eqref{eq:GEVP}  has eigenpairs $(\lambda_j, x_j)$ where
\begin{equation}  \label{eq:set3}
\lambda_{j+1} = \frac{ \alpha_0 + 2 \sum_{l=1}^{m} \alpha_l \cos(l j\pi h) }{ \beta_0 + 2 \sum_{l=1}^{m} \beta_l \cos(l j\pi h)  }, \ x_{j+1,k+1} = c \cos( j \pi kh ), \  j, k =0,1,\cdots, n-1
\end{equation}
 with $h = \frac{1}{n-1}$.
\end{theorem}

The matrices can be complex-valued. For example, let $\iota^2 = -1, n=5, m=2, \alpha_0 = 8+2\iota, \alpha_1 = 5 - \iota, \alpha_2 = 2 \iota, \beta_0 = 6, \beta_1 = 3\iota, \beta_2 = 1-\iota.$ Then, with the setting in Theorem \ref{thm:set3},  we have the following matrices
\begin{equation} \label{eq:noh}
\begin{aligned}
A & = 
\begin{bmatrix}
4+\iota & 5-\iota & 2\iota & 0 & 0 \\
5-\iota & 8+4\iota & 5-\iota & 2\iota \\
2\iota & 5-\iota & 8+2\iota & 5-\iota & 2\iota \\
0 & 2\iota & 5-\iota & 8+4\iota & 5-\iota \\
0 & 0 & 2\iota & 5-\iota & 4+\iota
\end{bmatrix}, 
B & = 
\begin{bmatrix}
3 & 3\iota & 1-\iota & 0 & 0 \\
3\iota & 7-\iota & 3\iota & 1-\iota \\
1-\iota & 3\iota & 6 & 3\iota & 1-\iota \\
0 & 1-\iota & 3\iota & 7-\iota & 3\iota \\
0 & 0 & 1-\iota & 3\iota & 3
\end{bmatrix}.
\end{aligned}
\end{equation}
By direct calculations, the eigenpairs  of \eqref{eq:GEVP} are 
\begin{equation}
\begin{aligned}
\lambda_{1,2,3,4,5} & = 2- \frac{\iota}{2}, \frac{7-3\iota + (6-5\iota)\sqrt{2}}{9},  \frac75 - \frac65\iota, -\frac58+\frac38\iota, \frac{7-3\iota - (6-5\iota)\sqrt{2}}{9},\\
x_1 & = (1, 1, 1, 1, 1)^T, \\
x_2 & = (1, \frac{1}{\sqrt{2}}, 0, -\frac{1}{\sqrt{2}}, -1)^T, \\
x_3 & = (1, 0, -1, 0, 1)^T, \\
x_4 & = (1, -1, 1, -1, 1)^T, \\
x_5 & = (-1, \frac{1}{\sqrt{2}}, 0, -\frac{1}{\sqrt{2}}, 1)^T,
\end{aligned}
\end{equation}
which verifies Theorem \ref{thm:set3}.

 Now, we define a Hankel matrix $H^{(\alpha,m)} = (H^{(\alpha,m)}_{j,k}) \in \mathbb{C}^{n\times n}$ with entries
\begin{equation} \label{eq:h4}
H^{(\xi,m)}_{j,k}  = H^{(\xi,m)}_{n-j+1,n-k+1} = 
\begin{cases}
\xi_{j+k-1}, \quad & k = 1, \cdots, m-j+1, \quad j =1, \cdots, m, \quad 1\le m\le n, \\
0, \quad & \text{otherwise}.
\end{cases}
\end{equation}
We then have the following set of analytical eigenpairs.

\begin{theorem}[Analytical eigenvalues and eigenvectors, set 4] \label{thm:set4}

Let $n\ge2, 1\le m \le n-1$ and 
$
A = T^{(\alpha,m)} + H^{(\alpha,m)}, B = T^{(\beta,m)} + H^{(\beta,m)}
$ 
with $T^{(\xi,m)}$ and $H^{(\xi,m)}$, $\xi = \alpha, \beta$,  defined in \eqref{eq:T} and \eqref{eq:h4}, respectively.
Assume that $B$ is invertible.  Then,  the GEVP  \eqref{eq:GEVP} has eigenpairs $(\lambda_j, x_j)$ where
\begin{equation} \label{eq:set4}
\lambda_{j+1} = \frac{ \alpha_0 + 2 \sum_{l=1}^{m} \alpha_l \cos(l j\pi h) }{ \beta_0 + 2 \sum_{l=1}^{m} \beta_l \cos(l j\pi h)  }, \quad \ x_{j+1,k} = c \cos\big( j \pi (k-\frac12) h \big)
\end{equation}
with $h = \frac{1}{n}, k =1,2,\cdots, n, \ j=0,1,\cdots,n-1.$
\end{theorem}

We present the following example. Let $n=4, m=2, \alpha_0 = 7, \alpha_1 = 5, \alpha_2 = 2, \beta_0 = 5, \beta_1 = 3, \beta_2 = 1.$ Then, with the setting in Theorem \ref{thm:set4}, we have the following matrices
\begin{equation}
\begin{aligned}
A & = 
\begin{bmatrix}
12 & 7 & 2 & 0 \\
7 & 7 & 5 & 2 \\
2 & 5 & 7 & 7 \\
0 & 2 & 7 & 12
\end{bmatrix}, \qquad
B & = 
\begin{bmatrix}
8 & 4 & 1 & 0 \\
4 & 5 & 3 & 1 \\
1 & 3 & 5 & 4 \\
0 & 1 & 4 & 8
\end{bmatrix}.
\end{aligned}
\end{equation}
By direct calculations, the eigenpairs of \eqref{eq:GEVP} are 
\begin{equation}
\begin{aligned}
\lambda_{1,2,3,4} & = \frac{21}{13}, \ \frac{5 + 4 \sqrt{2}}{7}, \ 1,  \ \frac{5 - 4 \sqrt{2}}{7}, \\
x_1 & = (1, 1, 1, 1)^T, \\
x_2 & = (1, \sqrt{2} - 1, 1 - \sqrt{2}, -1)^T, \\
x_3 & = (1, -1, -1, 1)^T, \\
x_4 & = (-1, 1+\sqrt{2}, -1-\sqrt{2}, 1)^T, 
\end{aligned}
\end{equation}
which verifies Theorem \ref{thm:set4}.

\begin{remark}
Theorems \ref{thm:set1}--\ref{thm:set4} give analytical eigenvalues and eigenvectors for four different sets of GEVPs. The eigenvalues are of the same form while the eigenvectors are of different forms. For a fix bandwidth $m$, the internal entries of the matrices defined in Theorems \ref{thm:set1}--\ref{thm:set4} are the same while the modifications to the boundary rows are slightly different. This small discrepancy leads to different eigenvectors. It is well-known that for a complex-valued Hermitian matrix, the eigenvalues are real and eigenvectors are complex. The matrices in \eqref{eq:noh} are complex-valued. They are symmetric but not Hermitian. The eigenvalues are complex while the eigenvectors are real. 
 This property is for these special matrices and it would be interesting to generalize this property to larger classes of matrices.
\end{remark}

\subsection{Corner-overlapped block-diagonal matrices}

In this section, we consider the following type of matrix, that is, $G^{(\xi)} = (G^{(\xi)}_{j,k})$ with
\begin{equation} \label{eq:gg}
G^{(\xi)} = 
\begin{bmatrix}
\xi_3 & \xi_1     \\
\xi_1  & \xi_0  & \xi_1  & \xi_2   \\
& \xi_1  & \xi_3 & \xi_1     \\
& \xi_2  & \xi_1  & \xi_0 & \xi_1  & \xi_2   \\
& & &  \xi_1  & \xi_3 & \xi_1    \\
 & & & \xi_2 & \xi_1  & \xi_0 & \cdots    \\
 & & & & & \vdots & \ddots    \\
& & & & &   &  \ & \xi_3   \\
\end{bmatrix}.
\end{equation}
We assume that the dimension of the matrix is $(2n+1) \times (2n+1)$. 
It is a block-diagonal matrix where the corners of blocks are overlapped. Therefore, we refer to this type of matrices as the corner-overlapped block-diagonal matrices.

In this section, we derive their analytical eigenpairs.
To illustrate the idea, we consider the following matrix
\begin{equation}  \label{eq:A5}
A = 
\begin{bmatrix}
\alpha_3  & \alpha_1   \\
\alpha_1  & \alpha_0 & \alpha_1  & \alpha_2   \\
&  \alpha_1  & \alpha_3 & \alpha_1     \\
&  \alpha_2  &\alpha_1  & \alpha_0 & \alpha_1  \\
&& & \alpha_1  & \alpha_3   \\
\end{bmatrix}
\end{equation}
and its EVP \eqref{eq:evp}.
A direct symbolic calculation leads to the analytical eigenvalues
\begin{equation} \label{eq:ev12345}
\begin{aligned}
\lambda_1 & = \alpha_3, \\
\lambda_{2,3} & = \frac{1}{2} \Big( \alpha_0 + \alpha_2 + \alpha_3 \pm \sqrt{12\alpha_1^2 + (\alpha_0 + \alpha_2 - \alpha_3)^2} \Big), \\
\lambda_{4,5} & = \frac{1}{2} \Big( \alpha_0 - \alpha_2 + \alpha_3 \pm \sqrt{4\alpha_1^2 + (\alpha_0 - \alpha_2 - \alpha_3)^2} \Big).
\end{aligned}
\end{equation}

Alternatively, to find its eigenvalues, we note that the first and the last rows of \eqref{eq:evp} lead to
\begin{equation*}
\begin{aligned}
\alpha_3 x_1 + \alpha_1 x_2 & = \lambda x_1, \\
\alpha_1 x_4 + \alpha_3 x_5 & = \lambda x_5.
\end{aligned}
\end{equation*}
On one hand, we solve these equations to arrive at 
\begin{equation} \label{eq:x2x4}
\begin{aligned}
\alpha_1 x_2 & = (\lambda - \alpha_3) x_1, \\
\alpha_1 x_4 & = (\lambda - \alpha_3) x_5,
\end{aligned}
\end{equation}
which is then substituted into the third equation in \eqref{eq:evp}  to get
\begin{equation} \label{eq:x135}
(\lambda - \alpha_3) (x_5 - x_3 + x_1) = 0.
\end{equation}
On the other hand, we substitute \eqref{eq:x2x4} and the third equation of \eqref{eq:evp} into the second and the fourth equations in \eqref{eq:evp}  to get
\begin{equation} \label{eq:x24}
\begin{aligned}
\big(2 \alpha_1^2 + \alpha_0 (\lambda - \alpha_3) - \lambda (\lambda - \alpha_3) \big) x_2 + \big( \alpha_1^2 + \alpha_2(\lambda - \alpha_3) \big) x_4 & = 0, \\
\big( \alpha_1^2 + \alpha_2(\lambda - \alpha_3) \big) x_2 + \big(2 \alpha_1^2 + \alpha_0 (\lambda - \alpha_3) - \lambda (\lambda - \alpha_3) \big) x_4 & = 0.
\end{aligned}
\end{equation}

We now see that the EVP \eqref{eq:evp} is decomposed into two subproblems; that is, one EVP and QEVP as follows
\begin{equation} \label{eq:evp1}
\tilde A x= \lambda x, \qquad \text{where} \qquad \tilde A = [\alpha_3]_{1\times 1}
\end{equation}
and
\begin{equation} \label{eq:qevp2}
\lambda^2 I x - \lambda B x- Cx =0,
\end{equation}
where $I$ is the identity matrix (we assume that the dimension of $I$ is adaptive to its occurrence, in this case, it is $2\times 2$), 
\begin{equation*} 
B = 
\begin{bmatrix}
\alpha_0 + \alpha_3 & \alpha_2 \\
\alpha_2 & \alpha_0 + \alpha_3
\end{bmatrix}, 
\quad \text{and} \quad
C = 
\begin{bmatrix}
2 \alpha_1^2 - \alpha_0 \alpha_3 & \alpha_1^2 - \alpha_2 \alpha_3 \\
\alpha_1^2 - \alpha_2 \alpha_3 & 2 \alpha_1^2 - \alpha_0 \alpha_3
\end{bmatrix}.
\end{equation*}

Both matrices $B$ and $C$ are symmetric. The EVP \eqref{eq:evp1} has an analytical eigenvalue $\lambda = \alpha_3$ which is one of the eigenvalues in \eqref{eq:ev12345}. The characteristic polynomial of the QEVP \eqref{eq:qevp2} is
\begin{equation} \label{eq:cqevp2}
\chi(\lambda) = \det( \lambda^2 I + \lambda B + C),
\end{equation}
which is a polynomial of order four. From fundamental theorem of algebra, it has four roots. 
It is easy to verify that $\lambda_j, j = 2,3,4,5$ given in \eqref{eq:ev12345} are the four roots of the equation $\chi(x) = 0.$ 

Now, we propose the following idea. Assuming that $\lambda \ne 0$, we rewrite the QEVP \eqref{eq:qevp2} as 
\begin{equation} \label{eq:evp1new}
\Xi  x = \lambda x, 
\end{equation}
where
\begin{equation}  \label{eq:evp2ref}
\Xi  = B + \frac{1}{\lambda} C = 
\begin{bmatrix}
\alpha_0 + \alpha_3 + \frac{2 \alpha_1^2 - \alpha_0 \alpha_3}{\lambda}  & \alpha_2 + \frac{\alpha_1^2 - \alpha_2 \alpha_3}{\lambda} \\[0.2cm]
\alpha_2 + \frac{\alpha_1^2 - \alpha_2 \alpha_3}{\lambda}  & \alpha_0 + \alpha_3 + \frac{2 \alpha_1^2 - \alpha_0 \alpha_3}{\lambda}  
\end{bmatrix},
\end{equation}
which is a symmetric matrix. For a fixed $\lambda$ the matrix $\Xi$ can be viewed as a constant matrix. We apply Theorem \ref{thm:set1} with bandwidth $m=1$ and dimension $n=2$. 
Thus, the eigenvalues of $\Xi$ satisfy 
\begin{equation} \label{eq:qevp2from1}
\lambda_j = \Big( \alpha_0 + \alpha_3 + \frac{2 \alpha_1^2 - \alpha_0 \alpha_3}{\lambda_j} \Big) + 2 \Big( \alpha_2 + \frac{\alpha_1^2 - \alpha_2 \alpha_3}{\lambda_j} \Big) \cos(j \pi/3), \qquad j = 1,2,
\end{equation}
which is rewritten as a quadratic form in terms of $\lambda_j$
\begin{equation} 
\lambda_j^2 - \big( \alpha_0 + \alpha_3 + 2\alpha_2 \cos(j \pi/3) \big) \lambda_j -  \big( 2 \alpha_1^2 - \alpha_0 \alpha_3 + 2(\alpha_1^2 - \alpha_2 \alpha_3) \cos(j \pi/3) \big) = 0,  \ j = 1,2.
\end{equation}
For $j=1$, $\cos( j \pi/3) = 1/2$ and we obtain the eigenvalues $\lambda_{2,3}$ as in \eqref{eq:ev12345} while for $j=2$, $\cos( j \pi/3) = -1/2$ and we obtain the eigenvalues $\lambda_{4,5}$  as in \eqref{eq:ev12345}. If $\lambda=0$, then \eqref{eq:evp2ref} is invalid and we note that a shift (divide \eqref{eq:evp1new} by $\lambda - \alpha$ for some non-zero $\alpha$) will lead to the same set of eigenvalues. 

We now generalize the matrix $A$ defined \eqref{eq:A5} and consider the GEVP. 
Based on the idea above, we have the following theorem which gives analytical solutions to a class of GEVPs with  corner-overlapped block-diagonal matrices.

\begin{theorem}[Analytical eigenvalues and eigenvectors, set 5] \label{thm:femp2}
Let $A = G^{(\alpha)}$ and $B = G^{(\beta)}$. Assume that $B$ is invertible.  Then, the GEVP  \eqref{eq:GEVP} has $2n+1$ eigenvalues 
\begin{equation} 
\lambda_{2n+1} = \frac{\alpha_3}{\beta_3}, \qquad \lambda_{2j-1,2j} = \frac{-\hat b \pm \sqrt{\hat b^2 - 4 \hat a \hat c}}{2\hat a}, \qquad j =1,2,\cdots, n.
\end{equation}
where
\begin{equation}
\begin{aligned}
\hat a & = \beta_0 \beta_3 - 2 \beta_1^2 + 2 (\beta_2 \beta_3 - \beta_1^2) \cos(j\pi h), \\
\hat b & = 4 \alpha_1 \beta_1 - \beta_0 \alpha_3 - \alpha_0 \beta_3 - 2(\beta_2 \alpha_3 - 2\alpha_1 \beta_1 + \alpha_2 \beta_3) \cos(j\pi h), \\
\hat c & = \alpha_0 \alpha_3 - 2 \alpha_1^2 + 2 (\alpha_2 \alpha_3 - \alpha_1^2) \cos(j\pi h). \\
\end{aligned}
\end{equation}
The corresponding eigenvectors are $x_j = (x_{j,1}, \cdots, x_{j, 2n+1})^T$ with
\begin{equation} 
\begin{aligned}
x_{2n+1,2j+1} & = c (-1)^j, \quad x_{2n+1,2j} = 0,\quad j = 1,\cdots, n, \\
x_{j, 2k+1} & = \frac{  \alpha_1 - \lambda_j \beta_1 }{\lambda_j \beta_3 - \alpha_3} (x_{j,2k} + x_{j,2k+2}), \\
x_{j,2k} & = c\sin( \lceil \frac{j}{2} \rceil \pi k h), \quad h = \frac{1}{n+1}, \quad j = 1, \cdots, 2n, \ k =0,1,\cdots, n,
\end{aligned}
\end{equation}
where $\lceil \cdot \rceil$ is the ceiling function.
\end{theorem}

\begin{proof}
For simplicity, we denote $(\lambda, x = (x_1, \cdots, x_{2n+1})^T)$ as a generic eigenpair of the GEVP \eqref{eq:GEVP}. We assume that $x_0 = x_{2n+2} = 0.$ One one hand, the $(2j+1)$-th row of \eqref{eq:GEVP} leads to
\begin{equation}
\alpha_1 x_{2j} + \alpha_3 x_{2j+1}  + \alpha_1 x_{2j+2} = \lambda (\beta_1 x_{2j} + \beta_3 x_{2j+1}  + \beta_1 x_{2j+2}), \qquad j = 0,1,\cdots, n,
\end{equation}
which is simplified to
\begin{equation} \label{eq:thm51}
( \alpha_1 - \lambda \beta_1) ( x_{2j} + x_{2j+2})  = (\lambda \beta_3 - \alpha_3 ) x_{2j+1}, \qquad j = 0,1,\cdots, n.
\end{equation}
Using \eqref{eq:thm51} recursively and $x_0 = x_{2n+2} = 0$, we calculate
\begin{equation} \label{eq:thm5a}
\begin{aligned}
0 & = ( \alpha_1 - \lambda \beta_1) x_{2n+2} \\
& = (\lambda \beta_3 - \alpha_3 ) x_{2n+1} - ( \alpha_1 - \lambda \beta_1) x_{2n} \\
& = \cdots \\
& = (\lambda \beta_3 - \alpha_3 ) \Big( \sum_{j=0}^n (-1)^{n-j} x_{2j+1} \Big).
\end{aligned}
\end{equation}

One the other hand, the $(2j+2)$-th row of \eqref{eq:GEVP} leads to
\begin{equation}
\begin{aligned}
\alpha_2 x_{2j} + \alpha_1 x_{2j+1}  & + \alpha_0 x_{2j+2}  + \alpha_1 x_{2j+3}  + \alpha_2 x_{2j+4} \\
& = \lambda (\beta_2 x_{2j} + \beta_1 x_{2j+1}  + \beta_0 x_{2j+2}  + \beta_1 x_{2j+3}  + \beta_2 x_{2j+4} ),
\end{aligned}
\end{equation}
where $j = 0,1,\cdots, n-1.$ Using \eqref{eq:thm51} and $x_0 = x_{2n+2} = 0$, this equation simplifies to
\begin{equation} \label{eq:thm5b}
\tilde \alpha_1 x_{2j} + \tilde \alpha_0 x_{2j+2}  + \tilde \alpha_1 x_{2j+4} = \lambda ( \tilde \beta_1 x_{2j} + \tilde \beta_0 x_{2j+2}  + \tilde \beta_1 x_{2j+4}), \quad j = 0,1,\cdots, n-1,
\end{equation}
where
\begin{equation} \label{eq:thm5p}
\begin{aligned}
\tilde \alpha_0 & = \alpha_0 + \frac{2 \alpha_1 ( \alpha_1 - \lambda \beta_1) }{\lambda \beta_3 - \alpha_3}, && \tilde \beta_0 = \beta_0 + \frac{2 \beta_1 ( \alpha_1 - \lambda \beta_1) }{\lambda \beta_3 - \alpha_3}, \\
\tilde \alpha_1 & = \alpha_2  + \frac{\alpha_1 ( \alpha_1 - \lambda \beta_1)}{\lambda \beta_3 - \alpha_3}, && \tilde \beta_1 = \beta_2 + \frac{\beta_1 ( \alpha_1 - \lambda \beta_1) }{\lambda \beta_3 - \alpha_3}.
\end{aligned}
\end{equation}

The GEVP \eqref{eq:GEVP} with the matrices defined in this theorem is then decomposed to a block form
\begin{equation} \label{eq:thm5d}
\begin{bmatrix}
A_{ee} & \mathbf{0} \\
A_{oe} & A_{oo}
\end{bmatrix}
\begin{bmatrix}
x_e \\
x_o 
\end{bmatrix}
=
\begin{bmatrix}
\mathbf{0} \\
\mathbf{0}
\end{bmatrix},
\end{equation}
where $x_e = (x_2, \cdots x_{2n} )^T$, $x_o = (x_1, \cdots x_{2n+1} )^T$, the first block $A_{ee}$ are formed from \eqref{eq:thm5b}, and the blocks $A_{oe}$ and $A_{oo}$ are formed from \eqref{eq:thm51}. 
The characteristic polynomial of \eqref{eq:thm5d} is
\begin{equation} \label{eq:thm5c}
\chi(\lambda) = \det( A_{ee}) \det (A_{oo}).
\end{equation}
The roots of $\chi(\lambda) = 0$ give the eigenvalues. Firstly,  using \eqref{eq:thm5a}, $\det (A_{oo})= 0$  leads to the eigenpair which we denote it as $(\lambda_{2n+1}, x_{2n+1})$ with the eigenvector $x_{2n+1} = (x_{2n+1,1}, \cdots,  x_{2n+1,2n+1})^T$ and
\begin{equation}
\lambda_{2n+1} = \frac{\alpha_3}{\beta_3}, \qquad x_{2n+1,2j+1} = c(-1)^j, \quad x_{2n+1,2j} = 0,\quad j = 1,\cdots, n.
\end{equation}
This eigenvalue $\lambda_{2n+1}$ is simple due to the nonzero block matrix $A_{oe}$.
Now, the first block part of \eqref{eq:thm5d} that leads to $A_{ee} x_e = 0$ can be written as a GEVP with tridiagonal matrices defined using parameters in \eqref{eq:thm5p}. Applying Theorem \ref{thm:set1} with $m=1$, we obtain that the eigenpairs $(\lambda_j, x_{j})$ with $x_j = (x_{j,2}, \cdots, x_{j,2n})^T$ and
\begin{equation} \label{eq:x2k}
\lambda_j = \frac{ \tilde \alpha_0 + 2 \tilde \alpha_1 \cos(j\pi h) }{\tilde \beta_0 + 2 \tilde \beta_1 \cos(j\pi h)  }, \quad x_{j,2k} = c\sin( j \pi k h), \quad h = \frac{1}{n+1}, \ j, k =1,2,\cdots, n.
\end{equation}

Using \eqref{eq:thm5p} and rearranging the index of the eigenpairs accordingly (due to the quadratic feature in the eigenvalue, one eigenpair $(\lambda_j, x_{j})$ becomes two eigenpairs $(\lambda_{2j-1}, x_{2j-1})$ and $ (\lambda_{2j}, x_{2j})$), we have 
\begin{equation*} 
\lambda_{2j-1,2j} = \frac{-\hat b \pm \sqrt{\hat b^2 - 4 \hat a \hat c}}{2\hat a}, \qquad j =1,2,\cdots, n,
\end{equation*}
where
\begin{equation*}
\begin{aligned}
\hat a & = \beta_0 \beta_3 - 2 \beta_1^2 + 2 (\beta_2 \beta_3 - \beta_1^2) \cos(j\pi h), \\
\hat b & = 4 \alpha_1 \beta_1 - \beta_0 \alpha_3 - \alpha_0 \beta_3 - 2(\beta_2 \alpha_3 - 2\alpha_1 \beta_1 + \alpha_2 \beta_3) \cos(j\pi h), \\
\hat c & = \alpha_0 \alpha_3 - 2 \alpha_1^2 + 2 (\alpha_2 \alpha_3 - \alpha_1^2) \cos(j\pi h). \\
\end{aligned}
\end{equation*}

Associated with the eigenvalue $\lambda_{2j-1}$ and $\lambda_{2j}$, the eigenvectors are then denoted as $x_{2j-1}$ and $x_{2j}$, respectively. From \eqref{eq:x2k}, we have
\begin{equation*}
x_{2j-1, 2k} =  x_{2j, 2k} = c\sin( j \pi k h), \quad j, k =1,2,\cdots, n.
\end{equation*}
The odd entries of the eigenvectors are given by \eqref{eq:thm51} as
\begin{equation*}
\begin{aligned}
x_{2j-1, 2k+1} & = \frac{  \alpha_1 - \lambda_{2j-1} \beta_1 }{\lambda_{2j-1} \beta_3 - \alpha_3} (x_{2k} + x_{2k+2}), \\
x_{2j, 2k+1} & = \frac{  \alpha_1 - \lambda_{2j} \beta_1 }{\lambda_{2j} \beta_3 - \alpha_3} (x_{2k} + x_{2k+2}), \qquad k = 0,1,\cdots, n, \quad j=1,2,\cdots,n.
\end{aligned}
\end{equation*}
This completes the proof.
\end{proof}

 We give an example which arises from the FEM discretization. 
The quadratic FEM of the  Laplacian eigenvalue problem with $n$ uniform elements on the unit interval $[0,1]$ leads to the stiffness and mass matrices \cite{hughes2012finite,strang1973analysis}
\begin{equation} \label{eq:km1dp2}
K = 
\frac{1}{h}
\begin{bmatrix}
\frac{16}{3} & -\frac{8}{3} &  \\[0.2cm]
-\frac{8}{3} & \frac{14}{3} & -\frac{8}{3} & \frac{1}{3} \\[0.2cm]
& -\frac{8}{3} & \frac{16}{3} & -\frac{8}{3} &  \\[0.2cm]
&   & \ddots & \ddots & \ddots   \\
&& & -\frac{8}{3} & \frac{16}{3}  \\
\end{bmatrix}, \qquad
M = h
\begin{bmatrix}
\frac{8}{15} & \frac{1}{15} &  \\[0.2cm]
\frac{1}{15} & \frac{4}{15} & \frac{1}{15} & -\frac{1}{30} \\[0.2cm]
& \frac{1}{15} & \frac{8}{15} & \frac{1}{15}&  \\[0.2cm]
&   & \ddots & \ddots & \ddots   \\
&& & \frac{1}{15} & \frac{8}{15}  \\
\end{bmatrix},
\end{equation}
which are of dimension ${(2n-1) \times (2n-1)}$. With these matrices in mind, we have the following analytical result.

Let $K$ and $M$ be defined in \eqref{eq:km1dp2}. Then, the GEVP  $K u = \lambda M u$ has analytical eigenvalues
\begin{equation*}
\lambda_j = 
\begin{cases}
\frac{4 \big( 13 + 2\cos(j \pi h) - \sqrt{124 + 112\cos(j \pi h) -11 \cos^2(j \pi h)} \big) }{3 - \cos(j \pi h) }n^2, \quad & j = 1,\cdots, n-1, \\
10 n^2, \quad & j =n, \\
\frac{4 \big( 13 + 2\cos(j \pi h) + \sqrt{124 + 112\cos(j \pi h) -11 \cos^2(j \pi h)} \big) }{3 - \cos(j \pi h) }n^2, \quad & j = n+1, \cdots, 2n-1, \\
\end{cases}
\end{equation*}
and
analytical eigenvectors  $ u_j = ( u_{j,1}, $ $ \cdots,   u_{j, 2n-1})^T$ where 
\begin{equation*} 
\begin{aligned}
 u_{j, 2k+1} & = \frac{ 40 + \lambda_j^h h^2 }{80 - 8 \lambda_j^h h^2} ( u_{j,2k} +  u_{j,2k+2}),  j \ne n, j = 1, \cdots, 2n-1, \ k =0,1,\cdots, n-1, \\
 u_{j,2k} & = 
\begin{cases}
c\sin( j \pi k h), \quad j = 1, \cdots, n-1, \ k =0,1,\cdots, n, \\
c\sin( (j-n) \pi k h), \quad j = n+1, \cdots, 2n-1, \ k =0,1,\cdots, n, \\
\end{cases} \\
 u_{n,2j+1} & = c (-1)^j, j = 0, 1,\cdots, n-1, \quad  u_{n,2j} = 0,\quad j = 1,\cdots, n-1. \\
\end{aligned}
\end{equation*}

Generalization of Theorem \ref{thm:femp2} is possible. The matrix defined in \eqref{eq:gg} is a corner-overlapped block-diagonal matrix where each block (except the first and last) is of dimension $3\times3$. One can generalize the block to be of dimension $k\times k$ where $k\ge3$. We give the following example where the block is $4\times 4$.
The cubic FEM  of the  Laplacian eigenvalue problem with $n$ uniform elements on the unit interval $[0,1]$ leads to the stiffness and mass matrices \cite{hughes2012finite,strang1973analysis}
\begin{equation} \label{eq:km1dp3}
\begin{aligned}
K & = 
\frac{1}{h}
\begin{bmatrix}
\frac{54}{5} & -\frac{297}{40} & \frac{27}{20}  \\[0.2cm]
-\frac{297}{40}  & \frac{54}{5} & -\frac{189}{40} \\[0.2cm]
\frac{27}{20} & -\frac{189}{40} & \frac{37}{5} & -\frac{189}{40} & \frac{27}{20} & -\frac{13}{40}  \\[0.2cm]
&& -\frac{189}{40} & \frac{54}{5} & -\frac{297}{40} & \frac{27}{20}  \\[0.2cm]
&& \frac{27}{20} &  -\frac{297}{40}  & \frac{54}{5} & -\frac{189}{40} \\[0.2cm]
&   & & & & \ddots & \ddots & \ddots  \\[0.2cm]
&   & & & &  \frac{27}{20} & -\frac{297}{40} &  \frac{54}{5}  \\
\end{bmatrix}, \\ 
M & = h
\begin{bmatrix}
\frac{27}{70} & -\frac{27}{560} & -\frac{3}{140}  \\[0.2cm]
-\frac{27}{560}  & \frac{27}{70}  & \frac{33}{560} \\[0.2cm]
-\frac{3}{140}  & \frac{33}{560}  & \frac{16}{105} & \frac{33}{560}  & -\frac{3}{140} & \frac{19}{1680}  \\[0.2cm]
&& \frac{33}{560} & \frac{27}{70}  & -\frac{27}{560} & -\frac{3}{140}  \\[0.2cm]
&& -\frac{3}{140}  &  -\frac{27}{560}  & \frac{27}{70}  & \frac{33}{560} \\[0.2cm]
&   & & & & \ddots & \ddots & \ddots  \\[0.2cm]
&   & & & & -\frac{3}{140}  & -\frac{27}{560}&  \frac{27}{70}   \\
\end{bmatrix}, \\ 
\end{aligned}
\end{equation}
which are of dimension ${(3n-1) \times (3n-1)}$. To find the analytical eigenvalues, we follow the same procedure as for quadratic FEM. By static condensation, the matrix eigenvalue problem GEVP $K  u = \lambda M  u$ is first rewritten to a block-matrix form
\begin{equation} \label{eq:p3eigdecomp}
\begin{bmatrix}
A_{tt} & \mathbf{0} \\
A_{ot} & A_{oo}
\end{bmatrix}
\begin{bmatrix}
 u^t \\
 u^o 
\end{bmatrix}
=
\begin{bmatrix}
\mathbf{0} \\
\mathbf{0}
\end{bmatrix},
\end{equation}
where $u^t = ( u_3,  u_6, \cdots,  u_{3n-3} )^T$, $ u^o = ( u_1,  u_2,  u_4,  u_5,  \cdots,  u_{3n-2},  u_{3n-1} )^T$. 
Similarly, $\det (A_{oo})=0$ leads to the eigenvalue $\lambda_n = 10n^2, \lambda_{2n} = 42n^2$. To find the other eigenvalues, we rewrite $A_{tt}  u^t = \mathbf{0}$ as a cubic EVP
\begin{equation} \label{eq:cevp}
\begin{aligned}
0 & = 
(\Lambda)^3 
\begin{bmatrix}
8 & 1 &  \\
1 & 8 & 1 &  \\
   & \ddots & \ddots & \ddots &  \\
  &  & 1 & 8 & 1  \\
&& & 1 & 8  \\
\end{bmatrix}
+ 30 (\Lambda)^2
\begin{bmatrix}
-36 & 1 &  \\
1 & -36 & 1 &  \\
   & \ddots & \ddots & \ddots &  \\
  &  & 1 & -36 & 1  \\
&& & 1 & -36  \\
\end{bmatrix} \\
&\quad  + 360(\Lambda)
\begin{bmatrix}
64 & 3 &  \\
3 & 64 & 3 &  \\
   & \ddots & \ddots & \ddots &  \\
  &  & 3 & 64 & 3  \\
&& & 3 & 64  \\
\end{bmatrix}
- 25200
\begin{bmatrix}
2 & -1 &  \\
-1 & 2 & -1 &  \\
   & \ddots & \ddots & \ddots &  \\
  &  & -1 & 2 & -1  \\
&& & -1 & 2  \\
\end{bmatrix},
\end{aligned}
\end{equation}
which is of dimension ${(n-1) \times (n-1)}$. Herein, $\Lambda=\lambda h^2$. Using the derivations from the proof of Theorem \ref{thm:femp2}, we obtain the eigenvalues that are the roots of the equation
\begin{equation}
( 4 + \zeta ) \Lambda^3 - 30( 18 - \zeta ) \Lambda^2 + 360 (32 + 3 \zeta ) \Lambda - 25200 (1 - \zeta ) = 0,
\end{equation}
where $\zeta = \cos(j\pi h), j=1,\cdots,n-1$. The eigenvectors can be obtained similarly.

\subsection{Polynomial eigenvalue problems (PEVPs)}
The PEVP is to find the eigenpairs $ (\lambda \in \mathbb{C}, x \in \mathbb{C}^n) $ such that 
\begin{equation} \label{eq:pep}
P(\lambda) x = 0 \qquad \text{with} \qquad P(\lambda) := \lambda^q A_q + \lambda^{q-1} A_{q-1} + \cdots + A_0,
\end{equation}
where $A_j \in \mathbb{C}^{n \times n}, j = 0,1,\cdots, q$.
When $q=1$ and $q=2$, the PEVP reduces to linear eigenvalue problem (LEVP) and QEVP, respectively. The LEVP is also referred to as the GEVP as discussed above.

We note that the nonlinear eigenvalue problem $A_{ee} x_e = 0$ defined in the proof of Theorem \ref{thm:femp2} is a QEVP. In particular, the problem can be written as follows
\begin{equation} \label{eq:qevp}
\lambda^2 B x+ \lambda C x + Dx =0,
\end{equation}
where $B =  T^{(\gamma,1)}$ with $\gamma_0 = \beta_0 \beta_3 - 2 \beta_1^2, \gamma_1 = \beta_2 \beta_3 - \beta_1^2$,  $C =  T^{(\gamma,1)}$ with $\gamma_0 = 4\alpha_1 \beta_1 - \beta_0 \alpha_3 - \alpha_0 \beta_3, \gamma_1 = 2\alpha_1 \beta_1 - \beta_2 \alpha_3 - \alpha_2 \beta_3$,  and $D =  T^{(\gamma,1)}$ with $\gamma_0 = \alpha_0 \alpha_3 - 2 \alpha_1^2, \gamma_1 = \alpha_2 \alpha_3 - \alpha_1^2$.
Herein, $ T^{(\gamma,1)}$ is defined as \eqref{eq:T}.
The analytical eigenpairs are then derived based on the Theorem \ref{thm:set1} with $m=1$. The analytical eigenvalues of the cubic EVP \eqref{eq:cevp} are derived in a similar fashion.  We generalize these results and give the following  result.

\begin{theorem}[Analytical solutions to PEVP]
 Let $n\ge2, 1\le m \le n-1$ and 
$
A_j = T^{(\alpha^{(j)},m)} - H^{(\alpha^{(j)},m)}
$ 
with $T^{(\xi,m)}$ and $H^{(\xi,m)}$, $\xi = \alpha^{(j)},j=0,1,\cdots,q,$  defined in \eqref{eq:T} and \eqref{eq:h1}, respectively.
Then, the PEVP \eqref{eq:pep} has eigenpairs $(\lambda_j, x_j)$ that $x_j$ are given in \eqref{eq:set1} and $\lambda_j$ are the roots of the polynomials
\begin{equation} 
\sum_{k=0}^q \lambda_j^k \Big(  \alpha_0^{(k)} +  2 \sum_{l=1}^{m}  \alpha_l^{(k)} \cos(l j\pi h)  \Big) =  0, \qquad j =1,2,\cdots, n.
\end{equation}
\end{theorem}

\begin{proof}
 
Following \cite[p. 515]{meyer2000matrix} and the generalization in Theorem \ref{thm:set1}, 
we seek eigenvectors with entries of the form $c\sin( j \pi k h)$.
Then each row of the PEVP \eqref{eq:pep}, 
\begin{equation*}
\sum_{s=0}^q \sum_{k=1}^n \lambda^s A_{s, ik} x_{j,k} = \sum_{s=0}^q \lambda^s  \sum_{k=1}^n A_{s, ik} x_{j,k}  = 0, \quad i=1,\cdots, n,
\end{equation*}
which boils down to
\begin{equation} \label{eq:ep}
\sum_{s=0}^q \lambda^s \Big(  \alpha_0^{(s)} +  2 \sum_{l=1}^{m}  \alpha_l^{(s)} \cos(l j\pi h)  \Big) = 0, \quad i=1,\cdots, n.
\end{equation}
We note that the above form is independent of the row number $i$ and the eigenvector index number $j$. 
This means that a root of \eqref{eq:ep} and a vector $x_j = (x_{j,1}, \cdots, x_{j,n})^T$ with $x_{j,k} = c\sin( j \pi k h), k=1,\cdots,n,$ always satisfy the PEVP. Hence, the root and $x_j$ give an eigenpair. 
This completes the proof.
\end{proof}

We remark that similar analytical results can be obtained if the Hankel matrix $H$ is defined  as \eqref{eq:h2}, \eqref{eq:h3}, and \eqref{eq:h4}. The eigenvectors are of the same form in each Theorem \ref{thm:set2}--\ref{thm:set4} and there are $n$ eigenvectors. There are $nq$ eigenvalues and the eigenvalues are roots of the $n$ $q$-th order polynomials. Two examples are given in the previous subsection, so we omit presenting more examples for brevity.

\subsection{Generalizations} \label{sec:gen}
We now consider some other generalizations. 
The simplest case is the constant scaling.
Let $c_1 \ne 0 , c_2 \ne 0$ be constants, then the GEVP $c_1 Ax = \lambda c_2 B x$ has eigenvalues $\lambda = \tilde \lambda c_2/c_1$ where $\tilde \lambda$ is an eigenvalue of the GEVP $ A x= \tilde \lambda B x$. The eigenvectors remain the same. This constant scaling has applications in various numerical spectral approximations. For example, for FEM, the scaling is in the form $(1/h) Ax = \lambda h B x$ while for FDM, the scaling is $c_1 = 1/h^2, c_2=1$ where $h$ is the size of a mesh (cf., \cite{strang1973analysis,strang1988linear}). 

Following the book of Strang \cite[Section 5]{strang1988linear}, one may generalize these results to powers and products of matrices. For EVP of the form \eqref{eq:evp}, the powers $A^k, k \in \mathbb{Z}$ has eigenvalues $\lambda^k_j$, where $\lambda_j, j= 1,\cdots, n$ are eigenvalues of \eqref{eq:evp}. The eigenvectors remain the same. For the EVPs $A x=\lambda x$ and $Bx = \mu x$, if $A$ commutes with $B$, that is, $AB=BA$, then the two EVPs have the same eigenvectors and the eigenvalues of $AB$ (or $BA$) are $\lambda \mu$. Additionally, in this case, $A+B$ has eigenvalues $\lambda + \mu$. 
 For GEVP of the form \eqref{eq:GEVP}, similar results can be obtained. If the matrices entries are commutative, then
$
 B^{-1}A = A B^{-1}.
$
The GEVP
$
 A^k x = \mu B^k x
$
 has eigenvalues $\mu = \lambda^k$ where $\lambda$ is an eigenvalue of $ A x= \lambda B x$. Similar results can be obtained for products and additions. 

We now consider the tensor-product matrices from the FEM-based discretizations (see, for example, \cite{calo2019dispersion,deng2018dmm}) of the Laplacian eigenvalue problem in multi-dimension. 
 We have the following result.

\begin{theorem}[Eigenvalues and eigenvectors for tensor-product matrices] \label{thm:tp2}
Let $(\lambda_j, x_j), j=1,\cdots, n,$ be the eigenpairs of the GEVP \eqref{eq:GEVP} and $(\mu_k, y_k), k=1,\cdots, m,$ be the eigenpairs of the GEVP $C  y = \mu D  y$. Then, the GEVP
\begin{equation} \label{eq:tp2}
(A \otimes D + B \otimes C)  z = \eta (B \otimes D) z
\end{equation}
has eigenpairs $(\eta_{(j,k)}, z_{(j,k)})$ with 
\begin{equation}
\eta_{(j,k)} = \lambda_j + \mu_k, \qquad  z_{(j,k)} = x_j \otimes y_k, \qquad j=1,\cdots, n, \  k=1,\cdots, m.
\end{equation}
\end{theorem}

\begin{proof}
Let $(\lambda, x)$ be a generic eigenpair of $A x= \lambda B x$ and $(\mu, y)$ be a generic eigenpair of $C  y = \mu D  y$. Let $ z =  x \otimes  y$, we calculate that 
\begin{equation}
\begin{aligned}
(A \otimes D + B \otimes C)  z & = (A \otimes D + B \otimes C) ( x \otimes  y) \\
& = A x \otimes D y + B x \otimes C y \\
& = \lambda B x \otimes D y + B x \otimes \mu D y \\
& = (\lambda + \mu) (B x \otimes D y) \\
& = (\lambda + \mu) (B \otimes D) ( x \otimes  y) \\
& = (\lambda + \mu) (B \otimes D)  z,
\end{aligned}
\end{equation}
which completes the proof.
\end{proof}

\begin{remark}
Once the two sets of the eigenpairs are found, either numerically or analytically,  the eigenpairs for the GEVP in the form \eqref{eq:tp2} can be derived. 
A FEM (FDM, SEM, or IGA) discretization of $-\Delta u = \lambda u$ on unit square domain with a uniform tensor-product mesh leads to the GEVP in the form of \eqref{eq:tp2}. For three- or higher- dimensional problems, this result can be generalized  in a similar fashion. 
\end{remark}

\section{Trigonometric identities} \label{sec:ti}
In this section, we derive some trigonometric identities based on the eigenvector-eigenvalue identity that was rediscovered and coined recently in \cite{denton2021eigenvectors}. The eigenvector-eigenvalue identity for the EVP \eqref{eq:evp} is (see \cite[Theorem 1]{denton2021eigenvectors})
\begin{equation} \label{eq:evi}
| x_{j,k}|^2  \prod_{l=1,l\ne j}^n (\lambda_j - \lambda_l ) = \prod_{l=1}^{n-1} (\lambda_j - \mu_l^{(k)}), \quad j,k = 1,\cdots, n,
\end{equation}
where $A$ is a Hermitian matrix with dimension $n$, $(\lambda_j, x_j), j=1,\cdots,n,$ are eigenpairs of $A x = \lambda x$ with normalized eigenvectors $x_j = (x_{j,1}, \cdots, x_{j, n})^T$,  and $\mu_l^{(k)}$ is an eigenvalue of $A^{(k)}  y = \mu^{(k)}  y$ with
$A^{(k)}$ being the minor of $A$ formed by removing the $k^{\text{th}}$ row and column. 
We generalize this identity for the GEVPs as follows.

\begin{theorem}[Eigenvector-eigenvalue identity for the GEVP] \label{thm:gevi}
Let $A$ and $B$ be Hermitian matrices with dimension $n \times n$. 
Assume that $B$ is invertible. 
Let $(\lambda_j, x_j), j=1,\cdots, n,$ be the eigenpairs of the GEVP \eqref{eq:GEVP} with normalized eigenvectors $x_j = (x_{j,1}, \cdots, x_{j, n})^T$. Then, there holds
\begin{equation} \label{eq:gevi}
| x_{j,k}|^2 \prod_{l=1,l\ne j}^n (\lambda_j - \lambda_l) = \frac{\prod_{l=1}^{n-1} \eta_l^{(k)} }{\prod_{l=1,l\ne j}^n \eta_l } \cdot  \prod_{l=1}^{n-1} (\lambda_j - \mu_l^{(k)}), \quad j,k = 1,\cdots, n,
\end{equation}
where  $\mu_l^{(k)}$ is an eigenvalue of $A^{(k)}  y = \mu^{(k)} B^{(k)}  y$ with
$A^{(k)}$ and $B^{(k)}$ being minors of $A$ and $B$ formed by removing the $k^{\text{th}}$ row and column, respectively, $\eta_l$ is an eigenvalue of $B  y = \eta  y$, and $\eta_l^{(k)}$ is an eigenvalue of $B^{(k)}  y = \eta^{(k)}  y$.
\end{theorem}

\begin{proof}
We follow the proof of \eqref{eq:evi} for $A x = \lambda x$  that uses perturbative analysis in \cite[Sect. 2.4]{denton2021eigenvectors} (first appeared in \cite{mukherjee1989two}). Firstly, since $B$ is invertible, $\det(B) \ne 0$ and hence $\eta_l \ne 0, l =1,\cdots, n$. Let $Q(\lambda)$ be the characteristic polynomial of the GEVP \eqref{eq:GEVP}. Then, 
\begin{equation} \label{eq:Qa}
Q(\lambda) = \det(\lambda B - A) = \det(B) \det(\lambda I  - B^{-1}A) = \det(B)  \prod_{l=1}^n (\lambda - \lambda_l).
\end{equation}
The derivative $Q'(\lambda_j)$ of $Q(\lambda)$ at $\lambda = \lambda_j$ is 
\begin{equation} \label{eq:gi1}
Q'(\lambda_j) = \det(B) \prod_{l=1, l \ne j}^n (\lambda_j - \lambda_l).
\end{equation}
Similarly, let $P^{(k)}(\mu^{(k)})$ be the characteristic polynomial of the GEVP $A^{(k)}  y = \mu^{(k)} B^{(k)}  y$. Then, 
\begin{equation} \label{eq:gi2}
P^{(k)}(\mu^{(k)}) =  \det(B^{(k)})  \prod_{l=1}^{n-1} (\mu^{(k)}  - \mu^{(k)}_l).
\end{equation}

Now, with the limiting argument, we assume that $A$ has simple eigenvalues. Let $\epsilon$ be a small parameter and we define the perturbed matrix
\begin{equation}
A^{\epsilon,k} = A + \epsilon e_k e_k^T, \qquad k = 1,\cdots, n,
\end{equation}
where $\{ e_k\}_{k=1}^n$ is the standard basis. The perturbed GEVP is defined as
\begin{equation} \label{eq:GEVPp}
A^{\epsilon,k}  x^\epsilon = \lambda^{\epsilon,k} B x^\epsilon.
\end{equation}
Using \eqref{eq:Qa} and cofactor expansion, the characteristic polynomial of this  perturbed GEVP can be expanded as
\begin{equation}
Q^\epsilon(\lambda) = \det (\lambda B - A^{\epsilon,k} ) = Q(\lambda)  - \epsilon P^{(k)}(\lambda) + \mathcal{O}(\epsilon^2).
\end{equation}
With $x_j$ being a normalized eigenvector, one has 
\begin{equation} \label{eq:normb}
x_j^T \cdot x_j  = 1, \qquad x_j^T B x_j  = \eta_j, \qquad j = 1, \cdots, n.
\end{equation}
Using this normalization, from perturbation theory  \cite{konstantinov2003perturbation},  the eigenvalue $\lambda^\epsilon_j$ of \eqref{eq:GEVPp} can be expanded as
\begin{equation}
\lambda^{\epsilon,k} =  \lambda_j + \frac{\epsilon}{\eta_j} |x_{j,k}|^2 + \mathcal{O}(\epsilon^2).
\end{equation}
Applying the Taylor expansion and $ Q(\lambda_j)=0$, we rewrite 
\begin{equation}
\begin{aligned}
0 = Q^\epsilon(\lambda^{\epsilon,k}) & = Q(\lambda^{\epsilon,k})  - \epsilon P^{(k)}(\lambda^{\epsilon,k}) + \mathcal{O}(\epsilon^2) \\
& = Q(\lambda_j) +  \frac{\epsilon}{\eta_j} |x_{j,k}|^2 Q'(\lambda_j) - \epsilon P^{(k)}(\lambda_j) + \mathcal{O}(\epsilon^2) \\
& = \frac{\epsilon}{\eta_j} |x_{j,k}|^2 Q'(\lambda_j) - \epsilon P^{(k)}(\lambda_j) + \mathcal{O}(\epsilon^2), \\
\end{aligned}
\end{equation}
which the linear term in $\epsilon$ leads to 
\begin{equation} \label{eq:gi3}
|x_{j,k}|^2 Q'(\lambda_j) = \eta_j P^{(k)}(\lambda_j).
\end{equation}

Applying \eqref{eq:gi1} and  \eqref{eq:gi2} with $\det(B) = \prod_{l=1}^n \eta_l$ and $\det(B^{(k)}) = \prod_{l=1}^{n-1} \eta_l^{(k)}$  to \eqref{eq:gi3} completes the proof.
\end{proof}

The identity \eqref{eq:gevi} can be rewritten in terms of the characteristic polynomials as \eqref{eq:gi3}.
A similar identity in terms of determinants, eigenvalues, and rescaled eigenvectors was presented in \cite[eqn. 18]{kausel2018normalized} for real-valued matrices. The identity \eqref{eq:gevi} is in terms of only eigenvalues and eigenvectors. 

Based on these two identities \eqref{eq:evi} and \eqref{eq:gevi}, one can easily derive the following trigonometric identities.
For EVP, using \eqref{eq:evi} and applying Theorem \ref{thm:set1} with $m=1$ and $B$ being an identity matrix, we have 
\begin{equation} \label{eq:ti31}
\frac{2}{n+1} \sin^2 \frac{k \pi}{n+1}  =  \dfrac{ \prod_{j=1}^{n-1} \Big(\cos \frac{k \pi}{n+1}  - \cos \frac{j \pi}{n} \Big)}{ \prod_{j=1,j\ne k}^{n} \Big(\cos \frac{k \pi}{n+1}  - \cos \frac{j \pi}{n+1} \Big) }, \ n \ge 2, \ k =1,\cdots, n.
\end{equation}
We note that this identity is independent of the matrix entries $\alpha_0$ and $\alpha_1$. The left hand side can be written in terms of a cosine function as $\frac{1}{n+1} ( 1 - \cos \frac{2 k \pi}{n+1})$ to have an identity in terms of only cosine functions.   For example, let $n=2, k=1$, then the identity boils down to
\begin{equation}
\frac12 = \frac23 \sin^2 \frac{\pi}{3} = \frac{\cos(\pi/3) - \cos(\pi/2)}{ \cos(\pi/3) - \cos(2\pi/3) } = \frac{1/2 - 0}{1/2 + 1/2} = \frac12.
\end{equation}

Similarly, we have for $n \ge 2, \ k =1,\cdots, n, l = 2, \cdots, n-1$ 
\begin{equation} \label{eq:ti3}
\frac{2}{n+1} \sin^2 \frac{k l \pi}{n+1}  =  \dfrac{ \prod_{j=1}^{l-1} \Big(\cos \frac{k \pi}{n+1}  - \cos \frac{j \pi}{l} \Big)  \prod_{j=l}^{n-1} \Big(\cos \frac{k \pi}{n+1}  - \cos \frac{(j-l+1) \pi}{n-l+1} \Big) }{ \prod_{j=1,j\ne k}^{n} \Big(\cos \frac{k \pi}{n+1}  - \cos \frac{j \pi}{n+1} \Big) }.
\end{equation}
If we introduce the notation that $\prod_{j=1}^0 (\cdot) = 1$, then \eqref{eq:ti31} can be written as \eqref{eq:ti3} with $l=1$ or $l=n$. 

For the GEVP, using \eqref{eq:gevi} and applying Theorem \ref{thm:set1} with $m=1$, we have 
for $n \ge 2, \ l, k =1,\cdots, n,$ 
\begin{equation} \label{eq:ti3g}
\frac{2}{n+1} \sin^2 \frac{k l \pi}{n+1}  =  \dfrac{ \prod_{j=1}^{n-1} \Big( \beta_0 + 2\beta_1 \cos \frac{j \pi}{n} \Big)}{ \prod_{j=1,j\ne k}^{n} \Big( \beta_0 + 2 \beta_1\cos \frac{j \pi}{n+1} \Big) } \cdot \frac{ \Pi_1 \Pi_2 }{\Pi_3},
\end{equation}
where
\begin{equation}
\begin{aligned}
\Pi_1 & = \prod_{j=1}^{l-1} \Big( \frac{\alpha_0 + 2 \alpha_1\cos \frac{k \pi}{n+1}  }{\beta_0 + 2 \beta_1\cos \frac{k \pi}{n+1} }
- \frac{\alpha_0 + 2 \alpha_1\cos \frac{j \pi}{l}  }{\beta_0 + 2 \beta_1\cos \frac{j \pi}{l} }    \Big), \\
\Pi_2 & = \prod_{j=l}^{n-1} \Big( \frac{\alpha_0 + 2 \alpha_1\cos \frac{k \pi}{n+1}  }{\beta_0 + 2 \beta_1\cos \frac{k \pi}{n+1} }
- \frac{\alpha_0 + 2 \alpha_1\cos \frac{(j-l+1) \pi}{n-l+1}  }{\beta_0 + 2 \beta_1\cos \frac{(j-l+1) \pi}{n-l+1} }    \Big), \\
\Pi_1 & = \prod_{j=1, j\ne k}^{n} \Big( \frac{\alpha_0 + 2 \alpha_1\cos \frac{k \pi}{n+1}  }{\beta_0 + 2 \beta_1\cos \frac{k \pi}{n+1} }
- \frac{\alpha_0 + 2 \alpha_1\cos \frac{j \pi}{n+1}  }{\beta_0 + 2 \beta_1\cos \frac{j \pi}{n+1} }    \Big).
\end{aligned}
\end{equation}

It is obvious that \eqref{eq:ti3g} reduces to \eqref{eq:ti3} when $B$ is an identity matrix (or multiplied by a nonzero constant). 

\begin{remark}
Other similar trigonometric identities can be established. Moreover, Theorems \ref{thm:set1}--\ref{thm:femp2} give various analytical eigenpairs. An application of the eigenvector-eigenvalue identity \eqref{eq:evi} along with these analytical results set up a system of equations governing the eigenvalues of the minors of the original matrices. Thus, the eigenvalues of these minors can be found by solving the system of equations. 
\end{remark}

\section{Concluding remarks} \label{sec:conclusion} 

We first remark that the ideas of finding analytical solutions to the GEVPs with Toeplitz-plus-Hankel and corner-overlapped block-diagonal matrices can be applied to solve other problems where a particular solution form is sought. 
Other applications include matrix representations of differential operators such as the Schrödinger operator in quantum mechanics \cite{deng2019isogeometric} and the $2n$-order operators \cite{deng2019optimal}. 
Moreover, the idea to solve QEVP can be applied to solve other nonlinear EVPs. 

The boundary modifications in the Toeplitz-plus-Hankel matrices give new insights for designing better numerical methods. For example, the high-order IGA (cf.,\cite{hughes2008duality}) produces outliers in the high-frequency region of the spectrum. A method which modifies the boundary terms to arrive at the Toeplitz-plus-Hankel matrices will be outlier-free \cite{deng2020boundary}.
For FDM, the structure of the Toeplitz-plus-Hankel matrices gives insights to the design better higher-order approximations near the domain boundaries. Lastly, we remark that the corner-overlapped block-diagonal matrices have applications in the FEMs and the discontinuous Galerkin methods.

\section*{Acknowledgments} 
The author thanks Professor Gilbert Strang for several discussions on the Toeplitz-plus-Hankel matrices and on the potential applications to the design of better numerical methods.


\bibliographystyle{siam}


%

\bibliography{ref}

\begin{thebibliography}{10}

\bibitem{andjelic2020some}
{\sc M.~An{\dj}eli{\'c} and C.~M. da~Fonseca}, {\em Some determinantal
  considerations for pentadiagonal matrices}, Linear and Multilinear Algebra,
  (2020), pp.~1--9.

\bibitem{barrera2017asymptotics}
{\sc M.~Barrera and S.~Grudsky}, {\em Asymptotics of eigenvalues for
  pentadiagonal symmetric {T}oeplitz matrices}, in Large Truncated Toeplitz
  Matrices, Toeplitz Operators, and Related Topics, Springer, 2017, pp.~51--77.

\bibitem{calo2019dispersion}
{\sc V.~Calo, Q.~Deng, and V.~Puzyrev}, {\em Dispersion optimized quadratures
  for isogeometric analysis}, Journal of Computational and Applied Mathematics,
  355 (2019), pp.~283--300.

\bibitem{chang2009exact}
{\sc H.-W. Chang, S.-E. Liu, and R.~Burridge}, {\em Exact eigensystems for some
  matrices arising from discretizations}, Linear algebra and its applications,
  430 (2009), pp.~999--1006.

\bibitem{ciarlet2002finite}
{\sc P.~G. Ciarlet}, {\em The finite element method for elliptic problems},
  SIAM, 2002.

\bibitem{da2007eigenvalues}
{\sc C.~Da~Fonseca}, {\em On the eigenvalues of some tridiagonal matrices},
  Journal of Computational and Applied Mathematics, 200 (2007), pp.~283--286.

\bibitem{da2019eigenpairs}
{\sc C.~M. Da~Fonseca and V.~Kowalenko}, {\em Eigenpairs of a family of
  tridiagonal matrices: three decades later}, Acta Mathematica Hungarica,
  (2019), pp.~1--14.

\bibitem{deng2018dmm}
{\sc Q.~Deng and V.~Calo}, {\em Dispersion-minimized mass for isogeometric
  analysis}, Computer Methods in Applied Mechanics and Engineering, 341 (2018),
  pp.~71--92.

\bibitem{deng2020boundary}
\leavevmode\vrule height 2pt depth -1.6pt width 23pt, {\em A boundary
  penalization technique to remove outliers from isogeometric analysis on
  tensor-product meshes}, arXiv preprint arXiv:2010.08159,  (2020).

\bibitem{deng2021outlier}
\leavevmode\vrule height 2pt depth -1.6pt width 23pt, {\em Outlier removal for
  isogeometric spectral approximation with the optimally-blended quadratures},
  arXiv preprint arXiv:2102.07543,  (2021).

\bibitem{deng2019isogeometric}
{\sc Q.~Deng, V.~Puzyrev, and V.~Calo}, {\em Isogeometric spectral
  approximation for elliptic differential operators}, Journal of Computational
  Science, 36 (2019), p.~100879.

\bibitem{deng2019optimal}
\leavevmode\vrule height 2pt depth -1.6pt width 23pt, {\em Optimal spectral
  approximation of 2n-order differential operators by mixed isogeometric
  analysis}, Computer Methods in Applied Mechanics and Engineering, 343 (2019),
  pp.~297--313.

\bibitem{denton2021eigenvectors}
{\sc P.~Denton, S.~Parke, T.~Tao, and X.~Zhang}, {\em Eigenvectors from
  eigenvalues: a survey of a basic identity in linear algebra}, Bulletin of the
  American Mathematical Society,  (2021).

\bibitem{fasino1988spectral}
{\sc D.~Fasino}, {\em Spectral and structural properties of some pentadiagonal
  symmetric matrices}, Calcolo, 25 (1988), pp.~301--310.

\bibitem{hughes2012finite}
{\sc T.~J. Hughes}, {\em The finite element method: linear static and dynamic
  finite element analysis}, Courier Corporation, 2012.

\bibitem{hughes2014finite}
{\sc T.~J.~R. Hughes, J.~A. Evans, and A.~Reali}, {\em Finite element and
  {NURBS} approximations of eigenvalue, boundary-value, and initial-value
  problems}, Computer Methods in Applied Mechanics and Engineering, 272 (2014),
  pp.~290--320.

\bibitem{hughes2008duality}
{\sc T.~J.~R. Hughes, A.~Reali, and G.~Sangalli}, {\em Duality and unified
  analysis of discrete approximations in structural dynamics and wave
  propagation: comparison of p-method finite elements with k-method {NURBS}},
  Computer methods in applied mechanics and engineering, 197 (2008),
  pp.~4104--4124.

\bibitem{kausel2018normalized}
{\sc E.~Kausel}, {\em Normalized modes at selected points without
  normalization}, Journal of Sound and Vibration, 420 (2018), pp.~261--268.

\bibitem{konstantinov2003perturbation}
{\sc M.~Konstantinov, D.~W. Gu, V.~Mehrmann, and P.~Petkov}, {\em Perturbation
  theory for matrix equations}, Gulf Professional Publishing, 2003.

\bibitem{losonczi1992eigenvalues}
{\sc L.~Losonczi}, {\em Eigenvalues and eigenvectors of some tridiagonal
  matrices}, Acta Mathematica Hungarica, 60 (1992), pp.~309--322.

\bibitem{meyer2000matrix}
{\sc C.~D. Meyer}, {\em Matrix analysis and applied linear algebra}, vol.~71,
  Siam, 2000.

\bibitem{mukherjee1989two}
{\sc A.~K. Mukherjee and K.~K. Datta}, {\em Two new graph-theoretical methods
  for generation of eigenvectors of chemical graphs}, in Proceedings of the
  Indian Academy of Sciences-Chemical Sciences, vol.~101, Springer, 1989,
  pp.~499--517.

\bibitem{patera1984spectral}
{\sc A.~T. Patera}, {\em A spectral element method for fluid dynamics: laminar
  flow in a channel expansion}, Journal of computational Physics, 54 (1984),
  pp.~468--488.

\bibitem{smith1985numerical}
{\sc G.~D. Smith}, {\em Numerical solution of partial differential equations:
  finite difference methods}, Oxford university press, 1985.

\bibitem{solary2013finding}
{\sc M.~S. Solary}, {\em Finding eigenvalues for heptadiagonal symmetric
  {T}oeplitz matrices}, Journal of Mathematical Analysis and Applications, 402
  (2013), pp.~719--730.

\bibitem{strang1988linear}
{\sc G.~Strang}, {\em Linear algebra and its applications}, Cole Thomson
  Learning Inc,  (1988).

\bibitem{strang1973analysis}
{\sc G.~Strang and G.~J. Fix}, {\em An analysis of the finite element method},
  vol.~212, Prentice-hall Englewood Cliffs, NJ, 1973.

\bibitem{strang2014functions}
{\sc G.~Strang and S.~MacNamara}, {\em Functions of difference matrices are
  {T}oeplitz plus {H}ankel}, siam REVIEW, 56 (2014), pp.~525--546.

\end{thebibliography}
\end{document}